\documentclass{article}
\pdfoutput=1

\usepackage{amsmath,amsthm,amssymb}
\input{xy}
\xyoption{all}

\usepackage{microtype} 
\usepackage[
pdfborderstyle={},
pdfborder={0 0 0},
pagebackref,
pdftex]{hyperref}

\renewcommand*{\backref}[1]{}
\renewcommand*{\backrefalt}[4]{
  \ifcase #1 
  [No citations.]
  \or [#2]
  \else [#2]
  \fi }

\let\originalleft\left
\let\originalright\right
\renewcommand{\left}{\mathopen{}\mathclose\bgroup\originalleft}
\renewcommand{\right}{\aftergroup\egroup\originalright}

\theoremstyle{plain}
\newtheorem{theorem}{Theorem}[section]
\newtheorem{lemma}[theorem]{Lemma}
\newtheorem{corollary}[theorem]{Corollary}

\newtheorem{proposition}[theorem]{Proposition}

\newtheorem{question}[theorem]{Question}

\newtheorem*{claim*}{Claim}
\newtheorem{claim}[theorem]{Claim}

\newtheorem*{case*}{Case}

\theoremstyle{definition}
\newtheorem{definition}[theorem]{Definition}

\newcommand{\reflem}[1]{Lemma~\ref{L:#1}}
\newcommand{\refthm}[1]{Theorem~\ref{T:#1}}
\newcommand{\refclm}[1]{Claim~\ref{C:#1}}
\newcommand{\refsec}[1]{Section~\ref{S:#1}}
\newcommand{\refprop}[1]{Proposition~\ref{P:#1}}

\newcommand{\Aut}{\mathop{\rm Aut}} 
\newcommand{\Id}{\mathop{\rm Id}} 
\newcommand{\from}{\colon} 
\newcommand{\st}{\mathbin{\mid}} 

\newcommand{\calF}{\mathcal{F}}
\newcommand{\calG}{\mathcal{G}}
\newcommand{\calP}{\mathcal{P}}
\newcommand{\calQ}{\mathcal{Q}}
\newcommand{\calR}{\mathcal{R}}
\newcommand{\HH}{\mathbb{H}}
\newcommand{\RR}{\mathbb{R}}

\def\Isom{\mathrm{Isom}}
\def\MF{\mathcal{MF}}
\def\Q{\mathcal{Q}}

\def\PMF{\mathbb P \mathcal{MF}}
\def\T{\mathcal{T}}
\def\Mod{\mathrm{Mod}}

\def\M{\mathcal{M}}
\def\C{\mathcal{C}}

\def\qi{\raise1.6ex\hbox{\tiny$A$,$B$}\mspace{-25mu}\asymp}

\title{Hyperbolic spaces in Teichm\"uller spaces\thanks{This work is
    in the public domain.  The first author was supported by NSF grants
    DMS 0905748 and DMS 1207183.  The second author was supported by EPSRC grant
    EP/I028870/1.}}  
\renewcommand\footnotemark{} 

\author{Christopher J.~Leininger and Saul Schleimer}

\begin{document}
\maketitle 

\abstract{We prove, for any $n$, that there is a closed connected
orientable surface $S$ so that the hyperbolic space $\HH^n$
almost-isometrically embeds into the Teichm\"uller space of $S$, with
quasi-convex image lying in the thick part.  As a consequence, $\HH^n$
quasi-isometrically embeds in the curve complex of $S$.}


\section{Introduction}

We denote the Teichm\"uller space of a surface $S$ by $\T(S)$, and the
$\epsilon$--thick part by $\T_\epsilon(S)$; see \refsec{teichmuller}.
An {\em almost-isometric embedding} of one metric space into another
is a $(1,C)$--quasi-isometric embedding, for some $C \geq 0$; see
\refsec{hyperbolic}.  Let $\HH^n$ denote hyperbolic $n$--space.  The
main result of this paper is the following.

\begin{theorem}
\label{T:main}
For any $n \geq 2$, there exists a surface of finite type $S$ and an
almost-isometric embedding
\[ 
\HH^n \to \T(S).
\]
Moreover, the image is quasi-convex and lies in $\T_\epsilon(S)$ for
some $\epsilon > 0$.
\end{theorem}

According to \refprop{isometric-teich} below, Theorem~\ref{T:main}
remains true if we replace ``surface of finite type'' with ``closed
surface''.  Our work is motivated, in part, by the following open
question (see \cite{FMcc} for the case $n = 2$).
\begin{question}
Does there exist a closed surface $S$ of genus at least $2$, a closed
hyperbolic $n$--manifold $B$ with $n \geq 2$, and an $S$--bundle $E$
over $B$ for which $\pi_1(E)$ is Gromov hyperbolic?
\end{question}

\noindent
To explain the relationship with our theorem, suppose that
\[ 
S \to E \to B 
\]
is an $S$--bundle over $B =\HH^n/\Gamma$, for some closed surface $S$
and some torsion free cocompact lattice $\Gamma < \Isom(\HH^n)$.
The monodromy is a homomorphism to the mapping class group of $S$,
$\rho \from \pi_1(B)= \Gamma \to \Mod(S)$.  The mapping class group
$\Mod(S)$ acts on $\T(S)$ by isometries with respect to the
Teichm\"uller metric, and according to work of Farb-Mosher \cite{FMcc}
and Hamenst\"adt \cite{hamenstadt}, $\pi_1(E)$ is $\delta$-hyperbolic
if and only if we can construct a $\Gamma$--equivariant
quasi-isometric embedding
\[ 
f \from \HH^n \to \T(S) 
\]
with quasi-convex image lying in $\T_\epsilon(S)$ for some $\epsilon >
0$; see also \cite{mjsardar}.  (In fact the $\Gamma$--equivariance and
quasi-isometric embedding assumptions imply that the image lies in
$\T_\epsilon(S)$.)


Our main theorem states that if we drop the assumption of
equivariance, then quasi-isometric embeddings with all the remaining
properties exist.  On the other hand, as was shown in \cite{clm}, one
can find cocompact lattices $\Gamma < \Isom(\HH^2)$ and
$\Gamma$--equivariant quasi-isometries into $\T(S)$ with image in
$\T_\epsilon(S)$---for these examples the image is not quasi-convex.

The main theorem for $n = 2$ also contrasts with the situation of
isometrically embedding hyperbolic planes in $\T(S)$.  More precisely,
every geodesic in $\T(S)$ is contained in an isometrically embedded
hyperbolic plane (with the Poincar\'e metric) called a Teichm\"uller
disk.  However, it is well-known that {\em no} Teichm\"uller disk lies
in any thick part---this follows from \cite{masurclosedtrajectories}
which guarantees that along a dense set of geodesic rays in the
Teichm\"uller disk the hyperbolic length of some curve on $S$ tends to
zero.

The curve complex of $S$ is a metric simplicial complex $\C(S)$ whose
vertices are isotopy classes of essential simple closed curves, and
for which $k+1$ distinct isotopy classes of curves span a $k$--simplex
if they can be realized disjointly.  In \cite{MM1}, Masur and Minsky
proved that $\C(S)$ is $\delta$--hyperbolic.  One of the key
ingredients in their proof is the construction of a coarsely Lipschitz
map $\T(S) \to \C(S)$.  The restriction of this map to any
quasi-convex subset of $\T_\epsilon(S)$ is a quasi-isometry (see for
example \cite[Lemma 4.4]{rafischleimer} or \cite[Theorem
7.6]{shadows}).  Composing the almost-isometry of Theorem~\ref{T:main}
with the map $\T(S) \to \C(S)$ we have the following corollary.

\begin{corollary} 
\label{C:C(S)} 
For every $n \geq 2$, there exists a surface of finite type $S$ and a
quasi-isometric embedding
\[ 
\HH^n \to \C(S).
\]
\end{corollary} 

The case of $n=2$ here can be compared to the result of Bonk and
Kleiner \cite{bonkkleiner} in which it is shown that every
$\delta$--hyperbolic group which is not virtually free contains a
quasi-isometrically embedding hyperbolic plane.  The assumption that
the group is not virtually free implies the existence of an arc in the
boundary.  According to \cite{gabai} (see also
\cite{leiningerschleimer,lms}) with the exception of a few small
surfaces, there are indeed arcs in the boundary of $\C(S)$.  In
\cite{bonkkleiner} however, essential use is made of the fact that
there is an action of the group, and so even in the case $n = 2$,
Corollary \ref{C:C(S)} does not follow from \cite{bonkkleiner}.

We now explain the idea for the construction in the case $n = 2$.
Given a closed Riemann surface $Z$ and a point $z \in Z$, the
Teichm\"uller space $\T(Z,z)$ is naturally a $\HH^2$--bundle
over $\T(Z)$; see \refsec{forget}.  Given a biinfinite geodesic
$\tau$ in $\T(Z)$, the preimage of $\tau$ in $\T(Z,z)$ is a
$3$--manifold.  The parameterization $t \mapsto \tau(t)$ lifts to a
flow on the preimage of $\tau$ for which the flow lines are geodesics
in $\T(Z,z)$.  The fiber over $\tau(0)$ admits a pair of transverse
$1$--dimensional singular foliations---these are naturally associated
to the vertical and horizontal foliations of the quadratic
differential defining $\tau$.  Any two flow lines meeting the same
nonsingular leaf of the vertical foliation are forward asymptotic.
Therefore, we have a $1$--parameter family of forward asymptotic
geodesics in $\T(Z,z)$.  We use this to define a map from $\HH^2$ to
$\T(Z,z)$: we pick a horocycle $C \subset \HH^2$ and send the pencil
of geodesics perpendicular to $C$ to our set of forward asymptotic
geodesics in $\T(Z,z)$.

At the beginning of \refsec{H^3} we give a brief explanation of how
this can be modified to give the construction for $n = 3$. The idea
for $n \geq 4$ is then a straightforward inductive construction.\\

\noindent {\bf Acknowledgements.} We thank Richard Kent for useful
conversations as well as having originally asked about the existence
of quasi-isometric embeddings of hyperbolic planes into $\C(S)$.  We
thank the referee for their comments. 

\section{Hyperbolic geometry} 
\label{S:hyperbolic}

Suppose that $(X, d_X)$ and $(Y, d_Y)$ are metric spaces.

\begin{definition}
\label{Def:AlmostIsom}
A map $F \from X \to Y$ is a {\em $K$--almost-isometric embedding} if
for all $x, x' \in X$ we have
\[
|d_X(x,x') - d_Y(F(x),F(x'))| \leq K.
\]
\end{definition}

\noindent
We use the {\em exponential model} for hyperbolic space: $\HH^n =
\RR^{n-1} \times \RR$ with length element
\[
ds^2 = e^{-2t}\left(dx_1^2 + \ldots + dx_{n-1}^2\right) + dt^2.
\]
For two points $p, q \in \HH^n$ we use $d_\HH(p,q)$ to denote the
distance between them.  The exponential model of hyperbolic space is
related to the upper-half space model $U = \RR^{n-1} \times
(0,\infty)$ by the map $\HH^n \to U$ given by $(x,t) \mapsto (x,e^t)$.
In the exponential model, for every $x \in \RR^{n-1}$ the path
$\eta_x(t) = (x,t)$ is a {\em vertical geodesic} and is parameterized
by arc-length.

\begin{lemma} 
\label{L:criteria}
Suppose $(X,d_X)$ is a geodesic metric space and $\delta, \epsilon, R
> 0$ are constants. Suppose $F \from \HH^n \to X$ is a function with
the following properties.
\begin{enumerate}
\item
\label{L:crit:geod} 
$F \circ \eta_x$ is a geodesic for all $x \in \RR^{n-1}$.
\item 
\label{L:crit:slim}
For distinct $x, x' \in \RR^{n-1}$ the geodesics $F \circ \eta_x$ and
$F \circ \eta_{x'}$ are two sides of an ideal $\delta$--slim triangle
in $(X,d_X)$.
\item 
\label{L:crit:unif}
For any $x, x' \in \RR^{n-1}$ if $e^{-t}|x-x'| < \epsilon$ then
$d_X(F(x,t),F(x',t)) \leq R$.
\item 
\label{L:crit:prop} 
If $(x_k,t_k),(x_k',t_k) \in \HH^n$ satisfy $\displaystyle{\lim_{k \to
\infty}e^{-t_k}|x_k - x_k'| = \infty}$, then \\ $\displaystyle{\lim_{k
\to \infty} d_X\left(F(x_k, t_k), F(x_k', t_k)\right) = \infty}$.
\end{enumerate}
Then there exists a constant $K$ so that $F$ is a $K$--almost
isometric embedding.
\end{lemma}

A useful consequence of Property~\ref{L:crit:unif} is that for any
$x, x', t \in \RR$ we have
\begin{equation} 
\label{E:alt3} 
d\left(F(x,t),F(x',t)\right) \leq \frac{R}{\epsilon}e^{-t}|x-x'| + R. 
\end{equation} 

\noindent
The remainder of this section gives the proof of \reflem{criteria}.
We begin by controlling how $F$ moves the centers of ideal triangles.
To be precise: Suppose that $T = \calP \cup \calQ \cup \calR \subset
\HH^n$ is an ideal triangle where $\calP$ and $\calQ$ are distinct
vertical geodesics.  Let $r$ denote the point of $\calR$ with maximal
$t$--coordinate.  We call $r$ the {\em midpoint} of $\calR$.  Thus $r$
serves as a center for $T$.  Define $x = x(\calP), x' = x(\calQ)$.

Observe, say from the upper-half space model, that for all $t \geq
t(r)$ we have
\begin{equation} 
\label{E:horocycledist} 
d_\HH\left((x,t), (x',t)\right) \leq e^{-t}|x-x'| \leq e^{-t(r)}|x-x'| = 2.
\end{equation}
Thus, by Inequality (\ref{E:alt3}) we have $d_X(F(x,t), F(x',t)) \leq
2R/\epsilon+R$.  Define $\Delta = \max \{ 3\delta, 2R/\epsilon + R \}$
and define the {\em displaced height} of $T$ to be
\[
h_T = h(T) = \min \Big\{ t \in \RR \,\, \Big| \,\, 
d_X\left(F(x,t), F(\calQ)\right) \leq \Delta 
\,\, \mbox{or} \,\,
d_X\left(F(\calP), F(x',t)\right) \leq \Delta 
\Big\}.
\]
It follows that $h(T) \leq t(r)$.  Note that for any vertical triangle
$T$, Property~\ref{L:crit:slim} implies that $h(T) > -\infty$.

\begin{claim}
\label{C:SlimAbove}
For any vertical triangle $T = \calP \cup \calQ \cup \calR \subset \HH^n$,
\[
d_X\left(F(x, h_T), F(x', h_T)\right) \leq 3 \Delta,
\]
where $x = x(\calP)$, $x' = x(\calQ)$.
\end{claim}

\begin{proof}
Breaking symmetry, in this setting, allows us to assume that there is
some $s \in \RR$ so that $d_X(F(x',s), F(x,h_T)) \leq \Delta$.  Let
$t' = \max\{s, t(r)\}$.  Using the triangle inequality,
Inequality~\ref{E:alt3} and Property~\ref{L:crit:geod} we have
\begin{align*}
t' - h_T & = d_X\left(F(x,t'),(x,h_T)\right) \\
          & \leq d_X\left(F(x,t'), F(x',t')\right) + d_X\left(F(x',t'), F(x',s)\right) 
                              + d_X\left(F(x',s), F(x,h_T)\right) \\
          & \leq (2R/\epsilon + R) + (t' - s) + \Delta \\
\intertext{and similarly}
t' - s    & \leq 2R/\epsilon + R + t' - h_T + \Delta. 
\end{align*}
Thus $|h_T - s| \leq 2R/\epsilon + R + \Delta$.  Another application
of the triangle inequality and Property~\ref{L:crit:geod} implies that
$d_X\left(F(x,h_T), F(x',h_T)\right) \leq 2R/\epsilon + R + 2\Delta \leq
3\Delta$, as desired.
\end{proof}

As mentioned above, for every vertical triangle $T$ we have $h(T) >
-\infty$ and hence $t(r) - h(T) < \infty$.  We now obtain a uniform
bound on this quantity.

\begin{claim}
\label{C:UniformBoundOnDisplacement}
There is a constant $C_0 = C_0(F)$ so that $t(r) - h(T) \leq C_0$ for all
vertical triangles $T \subset \HH^n$.  
\end{claim}

\begin{proof}
Suppose not.  Then we are given a sequence of vertical triangles $T_k
= \calP_k \cup \calQ_k \cup \calR_k$ where $t(r_k) - h(T_k)$ tends to
infinity with $k$.  Here $r_k$ is the midpoint of $\calR_k$, the
non-vertical side. Define $t_k = t(r_k)$, $h_k = h(T_k)$.  Define $x_k
= x(\calP_k)$, $x_k' = x(\calQ_k)$ to be the horizontal coordinates of the
vertical sides of $T_k$.

Note that by Equation (\ref{E:horocycledist})
\begin{flalign*}
  && e^{-t_k}|x_k - x_k'|  & = 2 & \\
\text{and so}
  && e^{-h_k} |x_k - x_k'| & = e^{-h_k} \cdot 2e^{t_k} = 2e^{t_k-h_k}. &
\end{flalign*}
Thus $e^{-h_k} |x_k - x_k'|$ tends to infinity with $k$.  From
Property~\ref{L:crit:prop} we deduce that the quantity
$d_X\left(F(x_k, h_k), F(x_k', h_k)\right)$ also tends to infinity
with $k$.  This last, however, contradicts \refclm{SlimAbove}.
\end{proof}

We give the proof of \reflem{criteria}.  Fix any $p, q \in \HH^n$.  If
$x(p) = x(q)$ then we are done by Property~\ref{L:crit:geod}.  Suppose
instead that $x(p) \neq x(q)$.  Let $\calP \cup \calQ \cup \calR$
denote the vertical triangle having vertical sides $\calP$ and $\calQ$
so that $x(\calP) = x(p)$, $x(\calQ) = x(q)$; let $r \in \calR$ be the
midpoint of the non-vertical side.  Define $C_1 = 2C_0 + 5\Delta + 1$.
There are now two cases to consider.

\begin{case*}
Suppose that $t(p) \geq h(T) - C_1$.  
\end{case*}

Let $p' \in \calP$ and $q' \in \calQ$ be the points with $t(p') =
t(q') = \max \{t(p), t(r)\}$.  Then by the triangle inequality and
Equation (\ref{E:horocycledist}) we have
\begin{align*}
d_\HH(p, q') & \leq d_\HH(p, p') + d_\HH(p', q') \\
             & \leq t(p') - t(p) + 2 \\
             & \leq t(r) - h(T) + C_1 + 2 \\
             & \leq C_0 + C_1 + 2.
\end{align*}
It follows that $d_\HH(p,q)$ is estimated by $d_\HH(q',q) = |t(q') -
t(q)|$ up to an additive error at most $C_0 + C_1 + 2$.  Appealing to
Property~\ref{L:crit:geod}, Inequality (\ref{E:alt3}), and the
triangle inequality we similarly have
\begin{align*}
d_X\left(F(p),F(q')\right)  
     & \leq d_X\left(F(p),F(p')\right) + d_X\left(F(p'),F(q')\right)\\
     & \leq t(p') - t(p) + 2R/\epsilon  + R \\
     & \leq C_0 + C_1 + 2R/\epsilon + R.
\end{align*}
Thus $d_X\left(F(p),F(q)\right)$ is estimated by
$d_X\left(F(q'),F(q)\right) = d_\HH(q',q)$ with an additive error at
most $C_0 + C_1 + 2R/\epsilon + R$.  This completes the proof in this
case.

\begin{case*}
Suppose that $t(p), t(q) \leq h(T) - C_1$.  
\end{case*}

In this case, since the triangle $T = \calP \cup \calQ \cup \calR$ is
slim in $\HH^n$, we find that that $d_\HH(p,q)$ is estimated by $t(r)
- t(p) + t(r) - t(q)$ up to an additive error of at most $2$.  We now
show that $d_X(F(p), F(q))$ is also estimated by the latter quantity,
with a uniformly bounded error.  Using Property \ref{L:crit:geod} and
Inequality (\ref{E:alt3}) deduce
\[
d_X\left(F(p),F(q)\right) 
   \leq t(r) - t(p) + 2R/\epsilon + R + t(r) - t(q).
\]

We now give a lower bound for $d_X\left(F(p), F(q)\right)$.  Recall
that $F(\calP)$ and $F(\calQ)$ are two sides of a $\delta$--slim
triangle in $X$.  Let $\calR_X$ be the third side of this triangle.
Since
\[
d_X\left(F(p), F(\calQ)\right),d_X\left(F(\calP), F(q)\right) 
   > \Delta \geq \delta
\]
it follows that there are points $p_X, q_X \in \calR_X$ so that
$d_X\left(F(p), p_X\right), d_X\left(q_X, F(q)\right) \leq \delta$.
Thus the distance $d_X(p_X, q_X)$ is an estimate for $d_X\left(F(p),
F(q)\right)$ with an additive error at most $2\delta$.

Define $a = (x, h_T), b = (x', h_T)$.  Again, as in the previous
paragraph, there are points $a_X, b_X \in \calR_X$ within distance
$\delta$ of $F(a), F(b)$.  Since $d_\HH(a,b) \leq 2(t(r) - h(T)) + 2$
we find 
\begin{align*}
d_X(a_X, b_X) & \leq 2\delta + 2(t(r) - h(T)) + 2R/\epsilon + R\\
              & \leq 2\delta + 2C_0 + 2R/\epsilon + R.
\end{align*}
Note that the geodesic segments $[p_X, a_X], [b_X, q_X] \subset
\calR_X$ have length at least $h(T) - t(p) - 2\delta$ and $h(T) - t(q)
- 2\delta$ respectively.  Each of these is greater than $C_1 -
2\delta$.

If $p_X \in [a_X, b_X]$ then $C_1 - 2\delta \leq 2\delta + 2C_0 +
2R/\epsilon + R$ and this is a contradiction.  Similarly, deduce $q_X
\not\in [a_X, b_X]$.  If $p_X = q_X$ then $d_X(F(p), F(q)) \leq
2\delta < \Delta$, contradicting our assumption that $t(p) < h(T)$.
Finally, if $p_X \in (b_X, q_X)$ then an intermediate value argument
using the fact that $\calR_X$ is a geodesic implies $d_X(F(p),
F(\calQ)) \leq 3\delta$, again a contradiction.  Similarly $q_X$ is
not in $(p_X, a_X)$.  Thus, $[p_X, a_X] \cap [b_X, q_X]$ is either
empty or is equal to $[a_X, b_X]$.  We deduce that
\begin{align*}
d_X(p_X, q_X) & \geq 2h(T) - t(p) - t(q) - 4\delta - 2\delta - 2C_0 - 2R/\epsilon - R\\
              & \geq 2t(r) - t(p) - t(q) - 7\Delta - 4C_0.
\end{align*}
The proof of \reflem{criteria} is complete. \qed

\section{Foliations and projections} 
\label{S:subsurfaces}

Let $Z$ be a closed surface of genus at least $2$ and ${\bf z}$ a set
of marked points.  A {\em measured singular foliation} $\calF$ on
$(Z,{\bf z})$ is a singular topological foliation so that
\begin{itemize}
\item
$\calF$ has only prong-type singularties, 
\item
all one-prong singularties of $\calF$ appear at points of ${\bf z}$,
and
\item 
$\calF$ is equipped with a transverse measure of full support.  
\end{itemize}
We refer the reader to \cite{FLP,masurinterval} for a detailed
discussion of measured foliations.  Two measured (respectively,
topological) foliations are {\em measure equivalent} (respectively,
{\em topologically equivalent}) if they differ by isotopy and
Whitehead moves.  We will only be concerned with those foliations
which appear as the vertical foliation for some meromorphic quadratic
differential on $Z$ (see Section \ref{S:definitions}).  Every measured
singular foliation is measure equivalent to such a foliation for a
fixed complex structure on $Z$; see \cite{hubbardmasur}.

The space of measure classes of measured foliation on $(Z,{\bf z})$ is
denoted by $\MF(Z,{\bf z})$ and its projectivization by $\PMF(Z,{\bf
z})$.  A measured foliation $\calF \in \MF(Z,{\bf z})$ is {\em
arational} if it has no closed leaf cycles.  We say that $\calF$ is
{\em uniquely ergodic} if whenever $\calF' \in \MF(Z,{\bf z})$ is
topologically equivalent to $\calF$, then $\calF$ and $\calF'$
project to the same point in $\PMF(Z,{\bf z})$.  Both of these notions
depend only on the topological classes of the foliations, and not the
transverse measures.

If $\calF$ is a measured foliation representing an element of
$\MF(Z)$, and ${\bf z} \subset Z$ is a set of marked points, then
$\calF$ also determines an element of $\MF(Z,{\bf z})$.  We note
that it is important in this case that $\calF$ be a foliation,
and not an equivalence class of foliations.  If $\calF$ is
arational as an element of $\MF(Z)$, and if ${\bf z} = \{z\}$ is a
single point, then $\calF$ is also arational as an element of
$\MF(Z,z)$; see \cite{leiningerschleimer}.

By a {\em strict subsurface} $Y \subset Z - {\bf z}$ we mean a
properly embedded surface with {\em nonempty} boundary and a set of
punctures, possibly empty, such that every component of $\partial Y$
is an {\em essential curve} in $Z - {\bf z}$; that is, homotopically
nontrivial and nonperipheral.  We also assume that $Y$ is not a sphere
with $k$ punctures and $j$ boundary components where $k + j = 3$.  We
will only refer to subsurfaces in one context, and that is as follows.
Given a pair of arational measured foliation $\calF, \calG \in
\MF(Z,{\bf z})$ and a proper subsurface $Y \subset Z - {\bf z}$, we
have the {\em projection distance}
\[ 
d_Y(\calF,\calG) \in \mathbb Z_{\geq 0}
\]
between $\calF$ and $\calG$ in $Y$.  This is the distance in the
arc-and-curve complex of $Y$ between the the subsurface projections of
$\calF$ and $\calG$ to $Y$.  For a detailed discussion, see
\cite{MM1,MM2}.  All we use is that $d_Y$ satisfies a triangle
inequality
\[ 
d_Y(\calF_1,\calF_2) 
  \leq d_Y(\calF_1,\calG) + d_Y(\calG,\calF_2)
\]
for all arational measured foliations $\calF_1,\calF_2,\calG \in
\MF(Z,{\bf z})$.  This relates to Teichm\"uller geometry by
\refthm{Rafi} below.

\section{Teichm\"uller spaces} 
\label{S:teichmuller}

Here we set notation and recall some basic properties of Teichm\"uller
space.  For background on Teichm\"uller space, we refer the reader to
any of \cite{ahlforsqc,gardiner,abikoff,IT}.

\subsection{Teichm\"uller space, quadratic differentials and geodesics} 
\label{S:definitions}

Given a closed Riemann surface $Z$ with a finite (possibly empty) set
of marked points ${\bf z} \subset Z$, let $\T(Z,{\bf z})$ denote the
{\em Teichm\"uller space} of equivalence classes of marked Riemann
surfaces
\[ 
\T(Z,{\bf z}) = \left\{ [f \from (Z,{\bf z}) \to (X,{\bf x})] \, \left| \,
\begin{array}{l} f \mbox{ is an orientation preserving  homeo-}\\
  \mbox{morphism to the Riemann surface } X \end{array}
\right. \right\}. 
\]
The equivalence relation is defined by
\[ 
\big( f \from (Z,{\bf z}) \to (X,{\bf x}) \big) 
   \sim \big( g \from (Z,{\bf z}) \to (Y,{\bf y}) \big)
\]
if $f \circ g^{-1} \from (Y,{\bf y}) \to (X,{\bf x})$ is isotopic (rel
marked points) to a conformal map.  If ${\bf z} = \emptyset$, then we
write $\T(Z) = \{[f \from Z \to X]\}$.

The {\em Teichm\"uller distance} on $\T(Z,{\bf z})$ is defined by
\[ 
d_\T \big( [f \from (Z,{\bf z}) \to (X,{\bf x})],
           [g \from (Z,{\bf z}) \to (Y,{\bf y})] \big)
  =  \inf \left\{ \left. \frac{1}{2} \log \left( K_h \right) 
                       \,\right|\, h \simeq f \circ g^{-1} \right\}
\]
where $K_h$ is the dilatation of $h$ and where $h \from (Y,{\bf y}) \to
(X,{\bf x})$ ranges over all quasi-conformal maps isotopic (rel marked
points) to $f \circ g^{-1}$. 

Given $\epsilon > 0$, the {\em $\epsilon$--thick part of Teichm\"uller
space} $\T_\epsilon(Z,{\bf z}) \subset \T(Z,{\bf z})$ is the set of
points $[f \from (Z,{\bf z}) \to (X,{\bf x})] \in \T(Z,{\bf z})$ where
the unique complete hyperbolic surface in the conformal class of
$X-{\bf x}$ has its shortest geodesic of length at least $\epsilon$.
When $\epsilon$ is understood from context we will simply refer to
$\T_\epsilon(Z,{\bf z})$ as the {\em the thick part of Teichm\"uller
space}.

Let $\T(Z,{\bf z}) \to \M(Z,{\bf z})$ denote the projection to moduli
space obtained by forgetting the marking
\[ [
f \from (Z,{\bf z}) \to (X,{\bf x})] \mapsto [(X,{\bf x})] 
\]
or, equivalently, by taking the quotient by the mapping class group.
Mumford's compactness criterion \cite{bersmumford} now implies: For
any $\epsilon > 0$, the thick part $\T_\epsilon(Z,{\bf z})$ projects
to a compact subset of $\M(Z,{\bf z})$.  Conversely, the preimage of
any compact subset of $\M(Z,{\bf z})$ is contained in
$\T_\epsilon(Z,{\bf z})$ for some $\epsilon > 0$.

Suppose $(X,{\bf x})$ is a closed Riemann surface with marked points
and $q \in \Q(X,{\bf x})$ is a unit norm, meromorphic quadratic
differential with all poles simple and contained in ${\bf x}$.  We
also use $q$ to denote the associated Euclidean cone metric on $X$.
We note that $\Q(X) \subset \Q(X,{\bf x})$, for any set of marked
point ${\bf x} \subset X$.  Given $q \in \Q(X)$ we view it as an
element of $\Q(X,{\bf x})$ whenever it is convenient.

Given $q \in \Q(X,{\bf x})$ and $t \in \RR$, let $g_t \from (X,{\bf
x}) \to (X_t,g_t({\bf x}))$ denote the $e^{2t}$--quasi-conformal
Teichm\"uller mapping defined by $(q,t)$.  Let $q_t \in
\Q(X_t,g_t({\bf x}))$ denote the terminal quadratic differential.  For
any point $p \in X$ which is not a zero or pole of $q$ we have a {\em
preferred coordinate} $z_0$ for $(X,q)$ near $p$ and preferred
coordinate $z_t$ for $(X_t,q_t)$ near $g_t(p)$.  In these coordinates
$q = dz_0^2$ and $q_t=dz_t^2$, and $g_t$ is given by $(u,v) \mapsto
(e^t u, e^{-t} v)$.  If we mark $(X,{\bf x})$ by $f \from (Z,{\bf z}) \to
(X,{\bf x})$, then setting $f_t = g_t \circ f$ we have
\[ 
\tau_q(t) = [f_t \from (Z,{\bf z}) \to (X_t,g_t({\bf x}))]
\]
being a Teichm\"uller geodesic through $[f \from (Z,{\bf z}) \to (X,{\bf
x})]$; note that every Teichm\"uller geodesic can be described in this
way.  The Teichm\"uller geodesic $\tau$ is {\em $\epsilon$--thick} if
the image of $\tau$ lies in $\T_\epsilon(Z,{\bf z})$.  We also simply
say a geodesic $\tau$ is {\em thick} if it is $\epsilon$--thick for
some $\epsilon > 0$.  A collection of geodesics $\{\tau_\alpha\}$ is
{\em uniformly thick} if there is an $\epsilon > 0$ so that each
$\tau_\alpha$ is $\epsilon$--thick.

Given $q \in \Q(X,{\bf x})$ we will let $\calF(q), \calG(q)$ denote
the vertical and horizontal foliations respectively; that is, the
preimage in preferred coordinates of the foliations of $\mathbb C$ by
vertical and horizontal lines.  For $q \in \Q(X,{\bf x})$ and $t \in
\RR$ consider the associated Teichm\"uller mapping $g_t \from (X,{\bf
x}) \to (X_t,g_t({\bf x}))$ as above.  If $c \from \RR \to X$ is
a nonsingular leaf of $\calF(q)$ parameterized by arc-length with
respect to the $q$--metric, then composing with $g_t$ we obtain a
nonsingular leaf of the vertical foliation for the terminal quadratic
differential $\calF(q_t)$,
\[ 
g_t \circ c \from \RR \to X_t.
\]
From the description of $g_t$ in local coordinates we see that this is
parameterized {\em proportional} to arc-length and, in fact, the
$q_t$--length is given by
\begin{equation} 
\label{E:exp-length}
\ell_{q_t}\left( g_t \circ c |_{[x,x']} \right) = e^{-t}|x' - x|.
\end{equation}

\subsection{Properties of Teichm\"uller geodesics} 
\label{S:geodesicprop}

Suppose $\tau = \tau_q$ is the Teichm\"uller geodesic determined by
$[f \from (Z,{\bf z}) \to (X,{\bf x})] \in \T(Z,{\bf z})$ and $q \in
\Q(X,{\bf x})$.  The forward asymptotic behavior of $\tau$ is
reflected in the structure of the vertical foliation $\calF(q)$.  For
us, the most important instance of this is a result of
Masur~\cite{masurhaus}.

\begin{theorem}[Masur] 
\label{T:Masur-ue}
If there exists $\epsilon > 0$ and $\{t_k\}_{k=1}^\infty$ such that
\begin{itemize}
\item
$t_k \to \infty$ as $k \to \infty$ and 
\item
$\tau_q(t_k) \in \T_\epsilon(Z,{\bf z})$ for all $k$ 
\end{itemize}
then $\calF(q)$ is arational and uniquely ergodic.
\end{theorem}

In particular, if $\tau_q$ is thick then both $\calF(q)$ and
$\calG(q)$ are uniquely ergodic.  We say a pair of arational
foliations $\calF$ and $\calG$ are {\em $K$--cobounded} if for all
strict subsurfaces $Y \subset X - {\bf x}$ we have $d_Y(\calF, \calG)
\leq K$.  A result of Rafi~\cite[Theorem~1.5]{rafi} relates the
thickness of a geodsic $\tau_q \subset \T$ to the coboundedness of the
associated vertical and horizontal foliations.

\begin{theorem}[Rafi] 
\label{T:Rafi}
For all $\epsilon > 0$ there exists $K > 0$ so that if $q \in
\Q(X,{\bf x})$ has $\tau_q$ being $\epsilon$--thick then $\calF(q)$
and $\calG(q)$ are $K$--cobounded.

Conversely, for all $K > 0$ there exists $\epsilon > 0$ so that if
$q \in \Q(X,{\bf x})$ has $\calF(q)$ and $\calG(q)$ being
$K$--cobounded then $\tau_q$ is $\epsilon$--thick.
\end{theorem}

\subsection{Forgetting the marked point: the Bers fibration}
\label{S:forget}

Suppose now that $Z$ is a closed surface and $z \in Z$ is a single
marked point; we use $(Z, z)$ to denote $(Z, \{z\})$.  Let $p \from
\widetilde Z \to Z$ denote the universal covering.  Given $[f \from (Z,z)
\to (X,f(z))]$ we can forget the marked point to obtain an element
$[f \from Z \to X] \in \T(Z)$.  This defines a holomorphic map
\[ 
\Pi \from \T(Z,z) \to \T(Z)
\]
called the {\em Bers fibration}~\cite{bersfiber}.  The fiber of this
map over $[f \from Z \to X]$ is holomorphically identified with $\widetilde
X$, the universal covering of $X$.  Moreover, this identification is
canonical, up to the action of the covering group on $\widetilde X$.

The projection of Teichm\"uller spaces $\Pi \from \T(Z,z) \to \T(Z)$
descends to a projection of moduli spaces $\hat \Pi \from \M(Z,z) \to
\M(Z)$.  The fiber of $\hat \Pi$ over $X \in \M(Z)$ is just
$X/\Aut(X)$ and this is compact.

Recall that puncturing a closed surface once increases the hyperbolic
systole.  (Lift to universal covers and apply the Schwarz-Pick lemma.)
It follows that the preimage of $\T_{\epsilon}(Z)$ by $\Pi^{-1}$ is
contained in $\T_{\epsilon}(Z,z)$.

By a theorem of Royden~\cite{royden} the Teichm\"uller metric agrees
with the Kobayashi metric on Teichm\"uller space.  Recall that the
inclusion of the universal covering $\widetilde X \to \T(Z,z)$ is a
holomorphic embedding \cite{bersfiber}.  Thus, if we give $\widetilde
X$ the Poincar\'e metric $\rho_0$ --- one-half the hyperbolic metric
--- then $(\widetilde X,\rho_0) \to (\T(Z,z), d_\T)$ is a contraction
\cite{kobayashihyperbolic}.  Kra \cite{kra} further proved the
following.

\begin{theorem}[Kra] 
\label{T:kra} 
There exists a homeomorphism $h \from [0,\infty) \to [0,\infty)$ so that for
any $[f \from Z \to X] \in \T(Z)$, and any $\tilde x_1,\tilde x_2 \in
\widetilde X \subset \T(Z,z)$, we have
\[ 
h(\rho_0(\tilde x_1,\tilde x_1)) \leq  d_\T(\tilde x_1,\tilde x_2)
  \leq \rho_0(\tilde x_1,\tilde x_2). 
\]
\end{theorem}

The function $h$ can be described concretely in terms of the solution
to a certain extremal mapping problem for the hyperbolic plane which
was solved by Teichm\"uller \cite{teichmuller} and Gehring
\cite{gehring}.  We will extend $h$ to a nondecreasing function, $h
\from \RR \to [0,\infty)$ by declaring $h(t) = 0$ for all $t
\leq 0$.

\subsection{Branched covers} 
\label{S:branched}

Here we use branched covers to induce maps on Teichm\"uller space. 

Suppose $P \from \Sigma \to Z$ is a branched cover, branched over some
finite set of points ${\bf z} \subset Z$.  Then any complex structure
on $Z$ pulls back to a complex structure on $\Sigma$, and thus induces
a map $P^* \from \T(Z,{\bf z}) \to \T(\Sigma)$.  Regarding
Teichm\"uller space as the space of marked Riemann surfaces,
$\T(Z,{\bf z}) = \{[f \from (Z,{\bf z}) \to (X,{\bf x})]\}$, the
embedding is described as follows.  The branched covering $P \from
\Sigma \to (Z,{\bf z})$ induces a branched covering $U \from \Omega
\to (X,{\bf x})$, for some Riemann surface $\Omega$, namely the
branched cover induced by the subgroup $(f \circ P)_*(\pi_1(\Sigma -
P^{-1}({\bf z}))) < \pi_1(X - {\bf x})$.  By construction, there is a
lift of the marking homeomorphism $\phi \from \Sigma \to \Omega$.
This is described by the following commutative diagram.
\[ 
\xymatrix{
\Sigma \ar[d]_P \ar[r]^\phi & \Omega \ar[d]^U\\
(Z,z) \ar[r]^f & (X,x).} 
\]
Then, we have
\[
P^*([f \from (Z,{\bf z}) \to (X,{\bf x})]) = [\phi \from \Sigma \to \Omega].
\]

We now give a well-known consequence of these definitions. 

\begin{proposition} 
\label{P:isometric-teich}
If $P \from \Sigma \to Z$ is nontrivially branched at every point of
$P^{-1}({\bf z})$, then $P^* \from \T(Z,{\bf z}) \to \T(\Sigma)$ is an
isometric embedding.  Moreover, for all $\epsilon > 0$ there exists
$\epsilon' > 0$ so that $P^*(\T_\epsilon(Z,{\bf z})) \subset
\T_{\epsilon'}(\Sigma)$.
\end{proposition}

\begin{proof} 
When $P$ is a covering then $P^*$ is an isometric embedding; see
\cite[Section 7]{rafischleimer}.  The proof is identical in the
presence of nontrivial branching, as a one-prong singularity at a
point of ${\bf z}$ lifts to a regular point or to a three-prong or
higher singularity.

Let $\widetilde \M(Z, {\bf z})$ be the quotient of $\T(Z,{\bf z})$ by
the group of mapping classes of $(Z,{\bf z})$ that lift to $\Sigma$.
Note that $\widetilde \M (Z,{\bf z}) \to \M(Z,{\bf z})$ is a finite
sheeted (orbifold) covering.  The embedding $P^* \from \T(Z,{\bf z})
\to \T(\Sigma)$ descends to a map $\widetilde \M (Z,{\bf z}) \to
\M(\Sigma)$, giving a commutative square.
\[ 
\xymatrix{ \T(Z,{\bf z}) \ar[r]^{P^*} \ar[d] & \T(\Sigma) \ar[d]\\
\widetilde \M(Z,{\bf z}) \ar[r] & \M(\Sigma)} 
\] 
By Mumford's compactness criteria \cite{bersmumford}, the image of
$\T_\epsilon(Z,{\bf z})$ in $\widetilde \M(Z,{\bf z})$ is compact, and
hence so is the image in $\M(\Sigma)$.  Appealing to Mumford's
criteria again (for $\M(\Sigma)$), it follows that for some $\epsilon'
> 0$ we have $P^*(\T_\epsilon(Z,{\bf z})) \subset
\T_{\epsilon'}(\Sigma)$.
\end{proof}


In general, for any branched cover $P \from \Sigma \to Z$, branched
over ${\bf z} \subset Z$, consider ${\bf \sigma} = P^{-1}({\bf z})$ as
a set of marked points on $\Sigma$.  Then again there is an isometric
embedding
\[ 
P^* \from \T(Z, {\bf z}) \to \T(\Sigma, {\bf \sigma}). 
\]

If ${\bf \omega} \subset {\bf \sigma}$ then define $\Pi_\omega \from
\T(\Sigma, {\bf \sigma}) \to \T(\Sigma,{\bf \omega})$ by forgetting
the points of ${\bf \sigma}$ not in ${\bf \omega}$.  When ${\bf
\omega}$ is empty we may omit the subscript.  In this notation, the
composition $\Pi \circ P^*$ gives the map of
\refprop{isometric-teich}.  So, if $P$ is non-trivially branched at
all points of ${\bf \sigma}$ then $\Pi \circ P^*$ is an isometric
embedding.  If $P$ is not branched at all points of ${\bf \sigma}$
then $\Pi \circ P^*$ fails to be an isometric embedding; however it
remains $1$--Lipschitz.

\begin{proposition} 
\label{P:1-lip-teich}
If $P \from \Sigma \to Z$ is branched over ${\bf z}$ and if 
${\bf \omega} \subset {\bf \sigma} = P^{-1}({\bf z})$ is any subset
then 
\[
\Pi_\omega \circ P^* \from \T(Z,{\bf z}) \to \T(\Sigma,{\bf \omega})
\]
is $1$--Lipschitz.
\end{proposition}

\begin{proof}
The Bers fibration is a holomorphic map~\cite{bersfiber} and, by
forgetting the points of ${\bf \sigma} - {\bf \omega}$ one at a time,
we see that $\Pi_{\bf \omega} \from \T(\Sigma, {\bf \sigma}) \to
\T(\Sigma,{\bf \omega})$ is a composition of holomorphic maps, hence
holomorphic.  In particular, because the Teichm\"uller metric agrees
with the Kobayashi metric \cite{royden}, it follows that $\Pi_{\bf
\omega}$ is $1$--Lipschitz \cite{kobayashihyperbolic}.  Since $P^*$ is
an isometric embedding, the composition is $1$--Lipschitz.
\end{proof}

\section{An inductive construction} 
\label{S:proof}

The proof of Theorem~\ref{T:main} is constructive, but also appeals to
an inductive procedure.  We begin by constructing the required
embedding of $\HH^2$ into some Teichm\"uller space as the base
case of the induction, then produce an embedding of $\HH^3$ into
some other Teichm\"uller space, then an embedding of $\HH^4$,
and so on.  All the main ideas and technical difficulties are present
in the construction of the embedding of $\HH^2$ and then the
embedding of $\HH^3$ from that of $\HH^2$.  The only
further complications which arise to describe the embedding of
$\HH^n$ from $\HH^{n-1}$ for $n \geq 4$ are in the
notation, which becomes increasingly messy as $n$ increases.  This is
due to the fact that the proof for $n$ really depends on the proof for
all $2 \leq k < n$ (rather than just $n-1$).  For this reason, we
carefully describe the cases $n=2$ and $n=3$, and sketch the general
inductive step indicating only those things that require modification.

\subsection{The hyperbolic plane case} 
\label{S:H^2}

Let $Z$ be a closed hyperbolic surface.  Let $q \in \Q(Z)$ be a
nonzero holomorphic quadratic differential on $Z$ so that the
associated Teichm\"uller geodesic $[g_t \from Z \to Z_t]$ is thick.  Write
$\calF = \calF(q)$ and $\calG = \calG(q)$ for the
vertical and horizontal foliations of $q$, respectively.  Next, let
$c \from \RR \to Z$ be a nonsingular leaf of $\calF$
parameterized by arc-length with respect to $q$ and let $z = c(0)$ be
a marked point on $Z$; see \refsec{teichmuller}.

Our goal is to construct an almost-isometric embedding
\[ 
{\bf Z} \from \HH^2 \to \T(Z,z). 
\]
We consider an isotopy $Z \times \RR \to Z$, written $(w,x)
\mapsto f^x(w)$, where $f^x \from Z \to Z$ is a homeomorphism for all $x \in
\RR$, $f^0$ is the identity and $f^x(z) = c(x)$ for all $x \in
\RR$.  We further assume that $f^x$ preserves $\calF$ for
all $x \in \RR$. 

We can construct such an isotopy by piecing together isotopies defined
on small balls.  More precisely, we start with some $\epsilon$--ball
around $z$, and construct a vector field tangent to $\mathcal F$
supported in the ball with with norm identically equal to $1$ on the
$\epsilon/2$ ball.  The flow for time $t \in (-\epsilon/2,\epsilon/2)$
is an isotopy of the correct form.  Now we repeat this for a ball
around $c(\epsilon/2)$.  Since the arc of $c$ from $z$ to any point
$c(x)$ is compact, we can cover it with finitely many such balls to
produce the required isotopy.

We think of the isotopy as ``pushing $z$ along $c$''.  This determines
the {\em horocyclic coordinate}
\[ 
\widetilde c \from \RR \to \T(Z,z) 
\]
given by
\[ 
\widetilde c (x) = [f^x \from (Z,z) \to (Z,c(x))].
\]
The image of $\widetilde c$ lies in the Bers fiber over the basepoint $[\Id
\from Z \to Z] \in \T(Z)$; the fiber is identified with the universal
cover $\widetilde Z$ of $Z$.  As such, we can identify $\widetilde c$ 
with a lift of $c$ to $\widetilde Z$ and write
\[ 
\widetilde c \from \RR \to \widetilde Z \subset \T(Z,z). 
\]

Applying the Teichm\"uller mapping $g_t \from Z \to Z_t$ determined by
$q$ and $t \in \RR$ gives the {\em height coordinate}.  These
coordinates together define ${\bf Z} \from \HH^2 \to \T(Z,z)$
where
\[ 
{\bf Z}(x,t) = [g_t \circ f^x \from (Z,z) \to (Z_t,g_t(c(x)))].
\]
Here we are using the coordinates $(x,t)$ on $\HH^2$ described
in \refsec{hyperbolic}.

Since the marking homeomorphisms are determined by $x$ and $t$, we
simplify notation and denote the values in Teichm\"uller space by
\begin{equation} 
\label{E:alternatemap}
{\bf Z}(x,t) = \widetilde c_t(x) = (Z_t,g_t(c(x))).
\end{equation}
We also write 
\[ 
{\bf Z}(x,0) = \widetilde c(x) = (Z,c(x)).
\]
As the notation suggests, $\widetilde c_t \from \RR \to
\widetilde Z_t \subset \T(Z,z)$ is a lift of $g_t \circ c \from \mathbb
R \to Z_t$ to the universal cover $\widetilde Z_t$, thought of as the
fiber over $[g_t \from Z \to Z_t]$.

\begin{theorem}
\label{T:H2_isometric}
The map ${\bf Z} \from \HH^2 \to \T(Z,z)$ is an almost-isometric
embedding.  Moreover, the image lies in the thick part and is
quasi-convex.
\end{theorem}

\begin{proof}
We verify the hypothesis of Lemma \ref{L:criteria} to prove that ${\bf
Z}$ is an almost-isometric embedding and, along the way, prove that
the image is quasi-convex and lies in the thick part.

First, fix any $x \in \RR$ so that $\eta_x(t) = (x,t)$ is a
vertical geodesic in $\HH^2$.  Then $t \mapsto {\bf Z} \circ
\eta_x(t) = {\bf Z}(x,t) = (Z_t,g_t(c(x)))$ is a Teichm\"uller
geodesic, and hence Property~\ref{L:crit:geod} of \reflem{criteria}
holds.  Furthermore, since $t \mapsto Z_t$ is a thick geodesic, we see
that $\{{\bf Z} \circ \eta_x(t)\}_{x \in \RR}$ are uniformly
thick geodesics.  That is, that the union of these geodesics, over all
$x \in \RR$, project into a compact subset of $\M(Z,z)$; namely,
the preimage of the compact subset of $\M(Z)$ containing the image of
$t \mapsto Z_t$ (see \refsec{forget}).  In particular, the image of
${\bf Z}$ lies in the thick part of $\T(Z,z)$.

For each $x \in \RR$, the geodesic ${\bf Z} \circ \eta_x$ is
defined by the quadratic differential $q \in \Q(Z)$ viewed as a
quadratic differential in $\Q(Z,c(x))$.  We denote the vertical and
horizontal foliations of $q \in \Q(Z,c(x))$ by $\calF^x$ and
$\calG^x$, respectively, and consider them as measured foliations
in $\MF(Z,z)$ by pulling them back via $f^x$.  Since $f^x$ preserves
$\calF$, it follows that $\calF^x = \calF^0 \in
\MF(Z,z)$ for all $x \in \RR$.

Now, since $t \mapsto Z_t$ is a thick geodesic, by \refthm{Masur-ue}
the foliations $\calF$ and $\calG$ are arational.  Puncturing an
arational foliation once gives an arational foliation in the punctured
surface.  Hence $\calF^x$ and $\calG^x$ are also arational for all
$x$.  Since $\calF^x = \calF^0$ for all $x \in \RR$ and since the
geodesics $\{{\bf Z} \circ \eta_x\}_{x \in \RR}$ are uniformly thick,
\refthm{Rafi} implies that there exists $K > 0$ so that the pairs
$(\calF^0,\calG^x) = (\calF^x, \calG^x)$ are $K$--cobounded for all
$x$.  By the triangle inequality (applied to each subsurface $Y$) we
see that for all $x, x' \in \RR$ the pair $(\calG^x, \calG^{x'})$ is
$2K$--cobounded (to see that $\calG^x$ and $\calG^{x'}$ are different
foliations, note that $(\calF^0,\calG^x)$ and $(\calF^0,\calG^{x'})$
define different geodesics ${\bf Z} \circ \eta_x$ and ${\bf Z} \circ
\eta_{x'}$, respectively).

Appealing to the other direction in \refthm{Rafi} the geodesic
$\Gamma^{x,x'}$, determined by $\calG^x$ and $\calG^{x'}$ for distinct
$x, x' \in \RR$, is uniformly thick, independent of $x$ and $x'$.
From this and \cite[Theorem 4.4]{shadows} it follows that there is a
$\delta > 0$ so that ${\bf Z}\circ \eta_x$, ${\bf Z} \circ \eta_{x'}$
and $\Gamma^{x,x'}$ are the sides of a $\delta$--slim triangle for
every pair of distinct points $x,x' \in \RR$, and hence
Property~\ref{L:crit:slim} of \reflem{criteria} holds.  From this, it
follows that ${\bf Z}(\HH^2)$ (is contained in and) has Hausdorff
distance at most $\delta$ from the union of the geodesics
\[ 
{\bf Z}(\HH^2) 
  \cup \left( \bigcup_{x \neq x' \in \RR} \Gamma^{x,x'} \right) 
= 
  \left( \bigcup_{x \in \RR} {\bf Z} \circ \eta_x \right) 
  \cup \left( \bigcup_{x \neq x' \in \RR} \Gamma^{x,x'}\right).
\]
This is precisely the {\em weak hull} of $\{\calG^x\}_{x \in
\RR} \cup \{\calF^0\} \subset \PMF(Z,z)$, and so according
to \cite[Theorem 4.5]{shadows}, this set, hence also ${\bf Z}(\mathbb
H^2)$, is quasi-convex (the assumption in \cite{shadows} that the
subset of $\PMF(Z)$ be closed was not used in the proof).


Finally, we must prove that Properties \ref{L:crit:unif} and
\ref{L:crit:prop} of Lemma \ref{L:criteria} hold.  For this we can
appeal directly to Theorem \ref{T:kra}.  More precisely, observe that
because $\{Z_t\}_{t \in \RR}$ lies in the thick part, the
pull-back of the flat metric on $\widetilde Z_t$ (which we also denote
$q_t$) is uniformly quasi-isometric to the Poincar\'e metric $\rho_0$
on $\widetilde Z_t$.  That is, there exist constants $A,B \geq 0$ so
that
\begin{equation} 
\label{E:flathyperbolic}
\frac{1}{A} \left( d_{q_t} (\widetilde z,\widetilde z') - B \right)
\leq \rho_0(\widetilde z,\widetilde z') \leq A \, d_{q_t}(\widetilde
z,\widetilde z') + B
\end{equation}
for all $t \in \RR$ and $\widetilde z,\widetilde z' \in Z_t$ (see for example \cite[Lemma 2.2]{FMcc}).

Applying \eqref{E:alternatemap}, the upper bound of \refthm{kra},
\eqref{E:flathyperbolic} and \eqref{E:exp-length}, in that order, we
find
\begin{align*}
d_\T\left({\bf Z}(x,t),{\bf Z}(x',t)\right) 
  & =    d_\T(\widetilde c_t(x),\widetilde c_t(x'))\\
  & \leq \rho_0(\widetilde c_t(x),\widetilde c_t(x'))\\
  & \leq A \, d_{q_t}(\widetilde c_t(x),\widetilde c_t(x')) + B\\
  & = A e^{-t} |x'-x| + B.
\end{align*}
So, setting $\epsilon = 1$ and $R = A + B$, Property~\ref{L:crit:unif}
of \reflem{criteria} holds.

On the other hand, \eqref{E:alternatemap}, the lower bound of Theorem \ref{T:kra}, monotonicity of $h$, and \eqref{E:exp-length} gives
\begin{align*}
d_\T\left({\bf Z}(x,t),{\bf Z}(x',t)\right) 
 & =    d_\T(\widetilde c_t(x),\widetilde c_t(x'))\\
 & \geq h(\rho_0(\widetilde c_t(x),\widetilde c_t(x')))\\
 & \geq h \left( \frac{1}{A} \left( d_{q_t}(\widetilde c_t(x),\widetilde c_t(x')) - B \right) \right)\\
 & = h \left( \frac{1}{A} (e^{-t}|x' - x| - B) \right).
\end{align*}
From this, and because $h$ is a homeomorphism on $[0,\infty)$ and
hence proper, Property~\ref{L:crit:prop} of Lemma \ref{L:criteria}
also holds.  This completes the proof of \refthm{H2_isometric}.
\end{proof}

\subsection{Hyperbolic $3$--space} 
\label{S:H^3}

Before diving into the construction, we explain the basic idea.  Our
embedding of the hyperbolic plane in \refsec{H^2} sends $(x,t)$ to
${\bf Z}(x,t) \in \T(Z,z)$ by pushing the marked point $z$ distance
$x$ along a leaf of the vertical foliation of a quadratic differential
then travelling distance $t$ along the Teichm\"uller flow.  There is a
simple extension of this construction which produces a map of
hyperbolic $3$--space into Teichm\"uller space $\T(Z,\{z,w\})$.  Take
$z$ and $w$ to lie on distinct leaves and send $(x,y,t)$ to the point
of $\T(Z,\{z,w\})$ obtained by pushing $z$ a distance $x$ along its
leaf, pushing $w$ a distance $y$ along its leaf, and applying the
Teichm\"uller flow for time $t$.

The problem is that whenever $z$ and $w$ move close to each other on
$Z$, the corresponding point in $\T(Z,\{z,w\})$ is in the thin part of
Teichm\"uller space; if $z$ and $w$ are very close to each other then
there is a simple closed curve surrounding $z$ and $w$ having an
annular neighborhood of large modulus.  This also shows that this map
$(x,y,t) \mapsto \T(Z,\{z,w\})$ is not a quasi-isometric embedding.
In fact the map is not even coarsely Lipschitz.

A more subtle construction is required.  We first choose a branched
cover $P \from \Sigma \to Z$, nontrivially branched at each point of
$P^{-1}(z)$.  According to \refprop{isometric-teich}, this induces an
isometric embedding of $\T(Z,z)$ into $\T(\Sigma)$.  Fix a suitably
generic point $w \in (Z, z)$ and pick a point $\sigma \in P^{-1}(w)$.
Roughly, we map our three parameters $(x,y,t)$ into
$\T(\Sigma,\sigma)$ as follows.  The coordinates $(x,t)$ determine
${\bf Z}(x,t) \in \T(Z,z)$ as in \refsec{H^2}.  The map $P^*$ applied
to ${\bf Z}(x,t)$ gives a point in $\T(\Sigma)$ as in
\refsec{branched}.  Finally use $y$ to determine a point ${\bf
\Sigma}(x,y,t) \in \T(\Sigma, \sigma)$, lying in the Bers fiber above
$P^* \circ {\bf Z}(x,t)$.  On its face, this new construction avoids
the problem we had before.  In $(Z, z)$ we have only one marked point;
after taking the branched covering over $z$ {\em we forget all of the
branch points over $z$}.  The single image of $\sigma$ can now move
freely enough so that we stay in the thick part of $\T(\Sigma,
\sigma)$.  We now explain this construction in more detail and prove
that the resulting map has all the required properties.

\subsubsection{The construction}

The notation from Section \ref{S:H^2} carries over to this section
without change.  Let $P \from \Sigma \to Z$ be a branched cover, branched
over the marked point $z \in Z$ so that $P$ is nontrivially branched
at every point of $P^{-1}(z)$.  This determines an isometric embedding
of Teichm\"uller spaces
\[ 
P^* \from \T(Z,z) \to \T(\Sigma)
\]
by Proposition \ref{P:isometric-teich}.  We write
\[ 
P^*([g_t \circ f^x \from (Z,z) \to (Z_t,g_t(c(x)))]) = [\phi_t^x \from \Sigma \to
  \Sigma_t^x] 
\]
so that $\phi_t^x$ is a lift of the marking $g_t \circ f^x$, and
$P_t^x$ is the induced branched cover making the following commute:
\[ 
\xymatrix{
\Sigma \ar[d]_P \ar[rr]^{\phi_t^x} & & \Sigma_t^x \ar[d]^{P_t^x}\\
(Z,z) \ar[rr]^{g_t \circ f^x} & & (Z_t,g_t(c(x))).} 
\]
The quadratic differentials $q_t$ pull back to quadratic differentials
$\lambda_t^x$ on $\Sigma_t^x$, and $g_t$ lifts to Teichm\"uller mappings
of the covers
\[ 
\psi_t^x \from \Sigma_0^x \to \Sigma_t^x
\]
so that $t \mapsto \Sigma_t^x$ is a Teichm\"uller geodesic for all
$x$.  The lifts satisfy $\phi_t^x = \psi_t^x \circ \phi_0^x$.  We have
another commutative diagram which may be helpful in organizing all the
maps:
\[ 
\xymatrix{
\Sigma \ar[d]_P \ar[rr]^{\phi_0^x} & & \Sigma_0^x 
          \ar[rr]^{\psi_t^x} \ar[d]^{P_0^x} & & \Sigma_t^x \ar[d]^{P_t^x}\\ 
(Z,z) \ar[rr]^{f^x} & & (Z,c(x)) 
          \ar[rr]^{g_t} & & (Z_t,g_t(c(x))).} 
\]

Denote the vertical foliation for $\lambda_t^x$ by $\Phi_t^x$.  Each
nonsingular leaf of $\Phi_t^x$ maps isometrically to a nonsingular
leaf of the vertical foliation $\calF_t$ for $q_t$ via the branched
covering $\Sigma_t^x \to Z_t$ since $\lambda_t^x$ is the pull back of
$q_t$.  Choose any nonsingular leaf $\gamma_0^0 \from \RR \to
\Sigma_0^0 = \Sigma$, parameterized by arc-length.  Observe that
$\gamma^0_0$ maps isometrically by $P$ to a leaf $\gamma \from \mathbb
R \to Z$ for $\calF$.  Note that $c$ and $\gamma$ are distinct leaves;
the preimage of $c$ in $\Sigma$ consists entirely of singular leaves,
namely the leaves that meet the branch points of $P$.

As we vary $x$, we can continuously choose lifts of $\gamma$ to leaves
$\gamma_0^x \from \RR \to \Sigma_0^x$ which agrees with our initial
leaf $\gamma_0^0$ when $x = 0$.  Specifically, we define the lift to
be
\[ 
\gamma_0^x 
 = \phi_0^x \circ \left( P|_{\gamma_0^0(\RR)}\right)^{-1} 
            \circ (f^x)^{-1} \circ \gamma.
\]
Composing with the lifts $\psi_t^x$, we obtain leaves $\gamma_t^x =
\psi_t^x \circ \gamma_0^x$.  Observe that via the branched covering
$P_t^x \from \Sigma_t^x \to Z_t$, $\gamma_t^x$ projects to the leaf $g_t
\circ \gamma$, independent of $x$.  Furthermore, this shows that the
$\lambda_t^x$--length of the arc $\gamma_t^x([y,y'])$ is the $q_t$--length
of $g_t \circ \gamma$ which is $e^{-t}|y-y'|$.

We pick a basepoint $\sigma = \gamma_0^0(0) \in \Sigma$, and consider
the surface $(\Sigma,\sigma)$, marked by the identity $\Id = \phi_0^0$
as a point in $\T(\Sigma,\sigma)$.  Just as we constructed $f^x$ by
pushing along $c$ to $c(x)$, we push $\sigma$ along $\gamma_t^x$ to
$\gamma_t^x(y)$ to obtain maps
\[ 
\xi_t^{x,y} \from (\Sigma,\sigma) \to (\Sigma_t^x,\gamma_t^x(y)).
\]
Specifically, we take $\xi_0^{x,y}$ to be the composition of $\phi_0^x$ and a map
isotopic to the identity on $\Sigma_0^x$ which preserves the foliation
$\Phi_0^x$ and pushes $\phi_0^x(\sigma)$ along $\gamma_0^x$ to
$\gamma_0^x(y)$.  Then $\xi_t^{x,y} = \psi_t^x \circ \xi_0^{x,y}$ maps
the foliation $\Phi_0^0$ to $\Phi_t^x$.

We denote the associated point in Teichm\"uller space
$[\xi_t^{x,y} \from (\Sigma,\sigma) \to (\Sigma_t^x,\gamma_t^x(y))] \in
\T(\Sigma,\sigma)$ simply by $(\Sigma_t^x,\gamma_t^x(y))$ as this
point is uniquely determined in this construction by $(x,y,t)$.

We define
\[ 
{\bf \Sigma} \from \HH^3 \to \T(\Sigma,\sigma) 
\]
in the coordinates $(x,y,t)$ for $\HH^3$ from Section
\ref{S:hyperbolic} by
\[ 
{\bf \Sigma}(x,y,t) = (\Sigma_t^x,\gamma_t^x(y)).
\]

\subsubsection{Fibration over $\HH^2$ case}

We also require a slightly different description of the map ${\bf
\Sigma}$ to take advantage of the construction in the $\HH^2$
case.  Observe that $P^* \circ {\bf Z} \from \HH^2 \to \T(Z,z) \to
\T(\Sigma)$ is an almost-isometric embedding, and is given by
\[ 
P^* \left( {\bf Z}(x,t) \right) = \Sigma_t^x,
\]
where $\Sigma_t^x$ denotes the point $[\phi_t^x \from \Sigma \to
\Sigma_t^x]$.  Recall that
\[
\Pi \from \T(\Sigma,\sigma) \to \T(\Sigma). 
\]
is the Bers fibration.  If we fix $(x,t) \in \HH^2$, then for
every $y$ we see that $(\Sigma_t^x,\gamma_t^x(y))$ is contained the
fiber $\Pi^{-1}(\Sigma_t^x)$.  Since $\Pi^{-1}(\Sigma_t^x)$ is
identified with the universal covering $\widetilde \Sigma_t^x$ of
$\Sigma_t^x$, just as in the case of $\HH^2$ we see that $t
\mapsto (\Sigma_t^x,\gamma_t^x(y))$ is a lift of $\gamma_t^x$ to
$\widetilde \Sigma_t^x \subset \T(\Sigma,\sigma)$.  As such, we use
the alternative notation
\[ 
\widetilde \gamma_t^x \from \RR \to \widetilde \Sigma_t^x \subset
\T(\Sigma,\sigma)
\]
with
\[ 
\widetilde \gamma_t^x(y) = (\Sigma_t^x,\gamma_t^x(y)) 
\]
when it is convenient to do so.

Finally we record the equation
\begin{equation} 
\label{E:fibration}
 \Pi \circ {\bf \Sigma}(x,y,t) = P^* \circ {\bf Z}(x,t)
\end{equation}
which holds for all $(x,y,t) \in \HH^3$.  The fact that $\Pi$ is
$1$--Lipschitz and $P^* \circ {\bf Z}$ is an almost-isometric
embedding provides us with useful metric information about ${\bf
\Sigma}$.

\begin{theorem}
\label{T:H3_isometric}
The map ${\bf \Sigma} \from \HH^3 \to \T(\Sigma, \sigma)$ is an
almost-isometric embedding.  Moreover, the image lies in the thick
part and is quasi-convex.
\end{theorem}

\begin{proof}
As before, we will verify the hypothesis of \reflem{criteria} to prove
that ${\bf \Sigma}$ is an almost-isometry and, along the way, prove
that the image is quasi-convex and lies in the thick part.

For all $(x,y) \in \RR^2$, the geodesic $\eta_{(x,y)}(t)$ in $\HH^3$
is sent to
\[ 
{\bf \Sigma} \circ \eta_{(x,y)}(t) = \left( \Sigma_t^x,\gamma_t^x(y) \right) 
  = \left( \psi_t^x(\Sigma_0^x),\psi_t^x(\gamma_0^x(y)) \right).
\]
This is a geodesic in $\T(\Sigma,\sigma)$ because $\psi_t^x \from
\Sigma_0^x \to \Sigma_t^x$ is a Teichm\"uller mapping; thus
Property~\ref{L:crit:geod} follows.  Furthermore, note that ${\bf
\Sigma} \circ \eta_{(x,y)}(t)$ lies over $P^*\circ {\bf Z} \circ
\eta_x(t)$ for all $(x,y,t)$.  Since $P$ is nontrivially branched over
every point, the uniform thickness of the set of geodesics $\{{\bf Z}
\circ \eta_x(t)\}_{x \in \RR}$ implies the same for $\{ P^* \circ {\bf
Z} \circ \eta_x(t)\}_{x \in \RR}$ by Proposition
\ref{P:isometric-teich}, and hence also for $\{ {\bf \Sigma} \circ
\eta_{(x,y)}(t) \st (x,y) \in \RR^2 \}$ by \eqref{E:fibration} as
discussed in Section \ref{S:forget}.  That is, ${\bf \Sigma}(\HH^3)$
lies in the thick part.

By our choice of maps $\xi_0^{x,y}$, if we pull back the vertical
foliation $\Phi_0^x$ of $\lambda_0^x$ to a foliation $\Phi_0^{x,y} \in
\MF(\Sigma,\sigma)$ the result is independent of $x$ and $y$.
Furthermore, Theorem \ref{T:Masur-ue} implies that these foliations,
as well as the pull backs of the horizontal foliations, are arational.
Thus all strict subsurface projection distances are defined.
\refthm{Rafi} and the results of \cite{shadows} can be applied as in
the $\HH^2$ case to prove that Property \ref{L:crit:slim} of
\reflem{criteria} is satisfied for some $\delta > 0$.  Furthermore,
${\bf \Sigma}(\HH^3)$ is quasi-convex.

We now come to the subtle point of the proof, which is verifying
Properties~\ref{L:crit:unif} and~\ref{L:crit:prop} of Lemma
\ref{L:criteria}.  We start with Property~\ref{L:crit:unif}.

\begin{claim*}
There exists $\epsilon > 0$ and $R > 0$ so that if $e^{-t}
\left| (x,y)-(x',y') \right| < \epsilon$ then
\[
d_\T \left( {\bf \Sigma}(x,y,t),{\bf \Sigma}(x',y',t) \right) < R. 
\]
\end{claim*}

Before we give the proof, we briefly explain the core technical
difficulty.  Fix $t$ and define $C_x = g_t(c(x))$ and $\Gamma_y =
g_t(\gamma(y))$.  Observe that, as before, when we vary $y$ we are
simply point pushing; thus the change in Teichm\"uller distance is
controlled by Theorem \ref{T:kra}.  On the other hand, varying $x$
means that we are varying the conformal stucture on the closed surface
$\Sigma_t^x$.  This is obtained by varying $x$ in $(Z_t, C_x)$ (which
is also point pushing) then taking a branched cover.  However, while
we vary $C_x$ in $Z_t$ we must also keep track of our $y$ coordinate,
which means we should also project $\gamma_t^x(y)$ down to
$Z_t$---this is precisely the point $\Gamma_y$.  Now if $C_x$ and
$\Gamma_y$ are close together and we vary $x$ so as to push these
points apart, then this can result in a large distance in the
``auxiliary'' Teichm\"uller space $\T(Z, \{z, w\})$, even for small
variation of $x$.  The idea is therefore to first vary $y$, if
necessary, to move $\gamma_t^x(y)$ in $\Sigma_t^x$ and so guaranteeing
that $\Gamma_y$ is not too close to $C_x$.  We can then vary $x$ as
required, then vary $y$ back to its original value.  Since the
variation of $y$ can be carried out independent of $x$, this will
result in a uniformly bounded change in Teichm\"uller distance.

\begin{proof}[Proof of Claim.]
Since the surfaces $\{\Sigma_t^x\}_{t,x \in \RR}$ lie in the
thick part, the (pulled back) metrics $\lambda_t^x$ and the Poincar\'e
metric(s) $\rho_0$ on the universal cover $\widetilde \Sigma_t^x$ are
uniformly comparable.  That is, there exist constants $A$ and $B$ so
that for all $\widetilde \sigma,\widetilde \sigma' \in \widetilde
\Sigma_t^x$
\begin{equation} 
\label{E:flathyperbolic2}
\frac{1}{A} \left( d_{\lambda_t^x} (\widetilde \sigma,\widetilde \sigma') - B \right) 
     \leq \rho_0 \left( \widetilde \sigma,\widetilde \sigma' \right) \leq A \, 
              d_{\lambda_t^x} \left( \widetilde \sigma,\widetilde \sigma' \right) + B.
\end{equation}
Applying Theorem \ref{T:kra}, Equations~\eqref{E:flathyperbolic2}
and~\eqref{E:exp-length} we have
\begin{align} 
\label{E:halfupper3}
d_\T \left( {\bf \Sigma}(x,y,t),{\bf \Sigma}(x,y',t) \right) 
 & =    d_\T \left( \widetilde \gamma_t^x(y),\widetilde \gamma_t^x(y') \right) \\ \nonumber
 & \leq \rho_0 \left( \widetilde \gamma_t^x(y),\widetilde \gamma_t^x(y') \right) \\ \nonumber
 & \leq A d_{\lambda_t^x} \left( \widetilde \gamma_t^x(y),\widetilde \gamma_t^x(y') \right) + B \\ \nonumber 
 & =    A \left( e^{-t} \left| y - y' \right| \right) + B. \nonumber
\end{align}

We now fix $t$ and the notation $C_x = g_t(c(x))$, $\Gamma_y =
g_t(\gamma(y))$.  To understand the effect of varying $x$ we must
consider the branched covering $P_t^x \from \Sigma_t^x \to (Z_t,
C_x)$, but also keep track of the image of our marked point
$\gamma_t^x(y) = \psi_t^x(\gamma_0^x(y))$ down in $(Z_t, C_x)$; that
is, the point $\Gamma_y$.  This results in the surface $Z_t$ with {\em
two marked points}:
\[ 
(Z_t, \{C_x, \Gamma_y\}).
\]
Appealing to Proposition \ref{P:1-lip-teich} we have
\begin{equation} 
\label{E:branchedforgetbounds}
d_\T \left( {\bf \Sigma}(x,y,t),{\bf \Sigma}(x',y',t) \right) 
   \leq d_\T \left( (Z_t, \{C_x,\Gamma_y\}), (Z_t,\{C_{x'},\Gamma_{y'}\}) \right). 
\end{equation}
This is because we are taking a branched covering, $\Sigma_t^x \to
Z_t$, and then forgetting all but one of the marked points in
$\Sigma_t^x$.

Since $Z_t$ lies in some fixed thick part of $\T(Z)$ for all $t \in
\RR$, there exists $\epsilon > 0$ so that the $2\epsilon$--ball
about $C_x$ in the $q_t$ metric, $B_{q_t}(C_x, 2 \epsilon)$ is a disk
for all $t,x \in \RR$ (that is, we have a lower bound on the
$q_t$--injectivity radius of $Z_t$, independent of $t$).  Now suppose
$\Gamma_y$ lies outside this ball
\[ 
\Gamma_y \not\in B_{q_t}(C_x, 2 \epsilon).
\]
Using again the fact that $Z_t$ lies in some thick part of $\T(Z)$ for
all $t \in \RR$, it follows that there is some $R' > 0$ with the
property that for any point $z' \in B_{q_t}(C_x, \epsilon)$ we have
\[ 
d_\T \left( (Z_t,\{C_x, \Gamma_y\}),(Z_t,\{z',\Gamma_y\}) \right) 
  < R'.
\]
Here the marking homeomorphism for $(Z_t,\{z',\Gamma_y\})$ is assumed
to differ from that of $(Z_t,\{C_x,\Gamma_y\})$ by composition with a
homeomorphism of $Z_t$ that is the identity outside
$B_{q_t}(C_x,2\epsilon)$.  In particular, if $e^{-t}|x-x'| < \epsilon$
and, crucially, $\Gamma_y \not\in B_{q_t}(C_x,2\epsilon)$ then deduce
that $C_{x'} \in B_{q_t}(C_x,\epsilon)$ and, from
Equation~\eqref{E:branchedforgetbounds}, that
\begin{align}  
\label{E:movexbound}
d_\T \left({\bf \Sigma}(x,y,t),{\bf \Sigma}(x',y,t) \right) 
  & \leq d_\T \left((Z_t,\{C_x,\Gamma_y\}),(Z_t,\{C_{x'}, \Gamma_y\}) \right)
    <    R'.
\end{align}

On the other hand, because the leaves of $\calF$ are geodesics for
$q_t$ and because $B_{q_t}(C_x, 2 \epsilon)$ is a disk,
if $\Gamma_y \in B_{q_t}(C_x, 2 \epsilon)$ then there exists $y' \in
\RR$ so that $e^{-t}|y'-y| \leq 2 \epsilon$ and
\[ 
\Gamma_{y'} \not\in B_{q_t}(C_x, 2\epsilon).
\]
Then, from \eqref{E:movexbound} we have
\begin{align*}
d_\T\left({\bf \Sigma}(x,y',t),{\bf \Sigma}(x',y',t) \right)  
 & \leq d_\T\left((Z_t, \{C_x,\Gamma_{y'}\}) , (Z_t, \{ C_{x'}, \Gamma_{y'}\})\right) 
   <    R'.
\end{align*}

Combining this, inequalities \eqref{E:halfupper3} and
\eqref{E:movexbound}, and the triangle inequality, it follows that for
any $x,y,x',t$ with $e^{-t}|x-x'| \leq \epsilon$ there is some $y' \in
\RR$ with $e^{-t}|y'-y| \leq 2 \epsilon$ such that
\begin{align} 
\label{E:bigmovexbound}
d_\T({\bf \Sigma}(x,y,t),{\bf \Sigma}(x',y,t)) 
 & \leq d_\T({\bf \Sigma}(x,y,t),{\bf \Sigma}(x,y',t)) + d_\T({\bf \Sigma}(x,y',t),{\bf \Sigma}(x',y',t)) \\ \nonumber
 & \phantom{\leq\,} + d_\T({\bf \Sigma}(x',y',t),{\bf \Sigma}(x',y,t)) \\ \nonumber
 & \leq 2(A(e^{-t}|y - y'|) + B) + R'\\ \nonumber
 & <    2(A2\epsilon + B)+R'\\ \nonumber
 & \leq 4(A \epsilon + B) + R'. 
\end{align}
Now, let $\epsilon> 0$ be as above and set $R = 5 (A \epsilon + B) +
R'$.  Given $(x,y,t),(x',y',t) \in \HH^3$ with
$e^{-t}|(x,y)-(x',y')| < \epsilon$, then we have
$e^{-t}|x-x'|,e^{-t}|y-y'| \leq e^{-t}|(x,y)-(x',y')| < \epsilon $.
Applying Equations~\eqref{E:halfupper3} and~\eqref{E:bigmovexbound}
and the triangle inequality we obtain
\begin{align*}
d_\T \left({\bf \Sigma}(x,y,t),{\bf \Sigma}(x',y',t)\right) 
 & \leq d_\T \left({\bf \Sigma}(x,y,t),{\bf \Sigma}(x',y,t)\right) 
        + d_\T\left({\bf \Sigma}(x',y,t),{\bf \Sigma}(x'y',t)\right) \\
 & \leq 4(A \epsilon + B) + R' + A \epsilon + B \\
 & <    5(A \epsilon + B) + R' = R.
\end{align*}
This completes the proof of the claim, and so verifies Property
\ref{L:crit:unif} of \reflem{criteria}.
\end{proof}

All that remains to show is Property~\ref{L:crit:prop} of
\reflem{criteria}.  Suppose we have a sequence of pairs
$\{(x_n,y_n,t_n),(x_n',y_n',t_n)\}_{n=1}^\infty$ with
$e^{t_n}|(x_n,y_n)-(x_n',y_n')| \to \infty$ as $n \to \infty$.  Then,
up to subsequence, we must be in one of two cases.

\begin{case*}
 $e^{t_n}|x_n - x_n'| \to \infty$ as $n \to \infty$.
\end{case*}

Forgetting the marked point is $1$--Lipschitz, and so we have
\begin{align*}
d_\T\left({\bf \Sigma}(x_n,y_n,t_n),{\bf \Sigma}(x_n',y_n',t_n)\right) 
 & \geq d_\T\left(\Sigma_{t_n}^{x_n},\Sigma_{t_n}^{x_n'}\right) \\
 & =    d_\T\left((Z_{t_n},g_{t_n}(x_n)),(Z_{t_n},g_{t_n}(x_n'))\right) \\
 & =    d_\T\left({\bf Z}(x_n,t_n),{\bf Z}(x_n',t_n)\right).
\end{align*}
However, we have already verified that ${\bf Z} \from \HH^2 \to
\T(Z,z)$ satisfies Lemma \ref{L:criteria}. Therefore the last
expression tends to infinity, and hence
\[ \lim_{n \to \infty} d_\T({\bf \Sigma}(x_n,y_n,t_n),{\bf \Sigma}(x_n',y_n',t_n)) = \infty\]
as required.

\begin{case*}
$e^{t_n}|y_n - y_n'| \to \infty$ as $n \to \infty$.
\end{case*}

If we also have $e^{t_n}|x_n - x_n'| \to \infty$, then we can appeal
to the previous case and we are done.  So we assume, as we may, that
$e^{t_n}|x_n-x_n'| < M$, for some constant $M > 0$.  Since we have
already shown that there are $\epsilon,R > 0$ so that part
\ref{L:crit:unif} from Lemma \ref{L:criteria} holds, it follows from \eqref{E:alt3}that
\[ 
d_\T\left( {\bf \Sigma}(x_n,y_n',t_n),{\bf \Sigma}(x_n',y_n',t_n)\right) 
 \leq \frac{R}{\epsilon} \left( e^{-t_n}|x_n - x_n'| \right) + R 
 \leq \frac{R}{\epsilon} M + R.
\]
Now, by the triangle inequality we have
\begin{align} 
\label{E:readyforkra}
d_\T \left( {\bf \Sigma}(x_n,y_n,t_n),{\bf \Sigma}(x_n',y_n',t_n) \right) 
 & \geq d_\T \left( {\bf \Sigma}(x_n,y_n,t_n),{\bf \Sigma}(x_n,y_n',t_n) \right) \\ \nonumber
 & \phantom{\geq\,} - d_\T \left( {\bf \Sigma}(x_n,y_n',t_n),{\bf \Sigma}(x_n',y_n',t_n) \right) \\ \nonumber
 & \geq d_\T \left( \widetilde{\gamma}_{t_n}^{x_n}(y_n),\widetilde{\gamma}_{t_n}^{x_n}(y_n') \right) 
     - \frac{R}{\epsilon}M - R \nonumber
\end{align}
We can now appeal to Theorem \ref{T:kra} as in our proof for ${\bf Z}
\from \HH^2 \to \T(Z,z)$ to find $A,B$ so that
\[
d_\T(\widetilde{\gamma}_{t_n}^{x_n}(y_n),\widetilde{\gamma}_{t_n}^{x_n}(y_n'))
\geq h \left( \frac{1}{A}e^{-{t_n}}|y_n' - y_n| - B\right)
\] 
The right-hand side tends to infinity by the properness of $h$, so we
can combine this with \eqref{E:readyforkra} to obtain
\[ 
\lim_{n \to \infty} d_\T({\bf \Sigma}(x_n,y_n,t_n),{\bf
  \Sigma}(x_n',y_n',t_n)) = \infty
\]
as required.  Therefore, Property~\ref{L:crit:prop} from
\reflem{criteria} holds, and the proof of \refthm{H3_isometric} is
complete.
\end{proof}

\subsection{The general case}

The previous arguments set up an inductive scheme for producing
almost-isometric embeddings of $\HH^n$ into Teichm\"uller
spaces.  The idea is as follows.

For $n-1 \geq 3$, induction gives us an almost-isometric embedding
${\bf W} \from \HH^{n-1} \to \T(W,w)$ satisfying all the
hypotheses of Lemma \ref{L:criteria} for some closed surface $W$ with
a marked point $w$.  We again take a branched cover
\[
P \from \Omega \to W
\]
nontrivially branched over each point in $P^{-1}(w) \subset \Omega$.
This determines a map
\[ 
P^*\circ {\bf W} \from \HH^{n-1} \to \T(\Omega). 
\]
Using the coordinates $(x, t) = (x_1, x_2, \ldots, x_{n-2}, t) \in
\HH^{n-1}$ we write
\[ 
P^* \circ {\bf W}(x,t) = \Omega_t^x.
\]

Inductively, we assume that the foliation of $\HH^{n-1}$ by asymptotic
geodesics $\{\eta_x(t)\}_{x \in \RR^{n-2}}$ are all mapped by ${\bf
W}$ to uniformly thick geodesics in $\T(W,w)$, so the same is true for
$P^* \circ W$.  These geodesics are obtained by applying the
Teichm\"uller mapping $\psi_t^x \from \Omega_0^x \to \Omega_t^x$
giving
\[ 
P^* \circ {\bf W} (x,t) = \psi_t^x \circ P^* \circ {\bf W}(x,0) 
\]
for all $x \in \RR^{n-2}$ and $t \in \RR$.  Furthermore, the defining
quadratic differentials all have the same vertical foliation.

We pick a leaf of this foliation $\gamma \from \RR \to \Omega$, and
arguing as before, this determines a leaf in each surface $\gamma_t^x
\from \RR \to \Omega_t^x$ with $\gamma_0^0= \gamma \from \mathbb R \to
\Omega_0^0 = \Omega$.  Now, pick $\omega = \gamma_0^0(0)$ to be our
marked point, add a factor of $\RR$ to $\HH^{n-1}$ with coordinate $y
= x_{n-1}$ to obtain $\HH^n$ with coordinates $(x, y, t) = (x_1,
\ldots, x_{n-2}, x_{n-1}, t)$, and define
\[ 
{\bf \Omega} \from \HH^n \to \T(\Omega,\omega)
\]
by
\[ 
{\bf \Omega}(x,y,t) = (\Omega_t^x,\gamma_t^x(y)).
\]
So we are again pushing a point along a leaf of the vertical foliation.

\begin{theorem}
\label{T:Hn_isometric}
The map ${\bf \Omega} \from \HH^n \to \T(\Omega, \omega)$ is an
almost-isometric embedding.  Moreover, the image lies in the thick
part and is quasi-convex.
\end{theorem}

\begin{proof}[Sketch of proof]
Again, we must verify the hypotheses of Lemma \ref{L:criteria} and
prove that the image of ${\bf \Omega}$ is quasi-convex in the thick
part, assuming that this is true in all previous steps of the
construction.

We can argue exactly as in the case of $\HH^3$ to prove
Properties~\ref{L:crit:geod} and~\ref{L:crit:slim} of
\reflem{criteria} as well as the fact that the image of ${\bf \Omega}$
is quasi-convex in the thick part.  Property~\ref{L:crit:unif}
requires more care.  However, once established,
Property~\ref{L:crit:prop} follows formally, just as in the case of
$\HH^3$.

We elaborate on the proof that Property~\ref{L:crit:unif} holds for
some $\epsilon$ and $R$.  For this, we must give a more precise
description of the construction.  Write $\Omega_{n-1} = \Omega$,
$\Omega_{n-2} = W$ and
\[ 
P_{n-2} = P \from \Omega_{n-1} \to \Omega_{n-2}
\]
for the branched cover used in the construction.  Inductively, we have
a tower of branched covers
\[ 
\xymatrix{ \Omega_{n-1} \ar[r]^{P_{n-2}} & \Omega_{n-2}
  \ar[r]^{P_{n-3}} & \cdots \ar[r] & \Omega_2 \ar[r]^{P_1} & \Omega_1}
\]
In this tower, $P_j$ is nontrivially branched at every point
$P_j^{-1}(\omega_j)$ where $\omega_j \in \Omega_j$ is the marked
point.  To clarify, we note that $\Omega_1 = Z$, $\omega_1 = z$,
$\Omega_2 = \Sigma$ and $\omega_2 = \sigma$ from the preceding
discussion.

We also have a quadratic differential $\nu_1$ on $\Omega_1$ (this is
$\nu_1 = q$ from before), which pulls back via all the branched covers
to quadratic differentials $\nu_i = P_{i-1}^*(\nu_{i-1}) \in
\Q(\Omega_i)$.  On $\Omega_1$, we have chosen $n-1$ distinct
nonsingular leaves from the vertical foliation of $\nu_1$ which we
denote $\{\zeta_i \from \RR \to \Omega_1\}_{i=1}^{n-1}$.  These
leaves are parametrized by arc-length so that $\zeta_j(0) = P_1 \circ
P_2 \circ \cdots \circ P_{j-1}(\omega_j)$.  

Recall that $y = x_{n-1}$.  We can now describe $\Omega(x, y, t) =
\Omega(x_1,\ldots,x_{n-2},x_{n-1},t)$ for any $(x,y,t) \in \HH^n$.
At the bottom of the tower we push $\omega_1$ along $\zeta_1$ to
$\zeta_1(x_1)$, then take the branched cover $\Omega_2^{x_1} \to
(\Omega_1,\zeta_1(x_1))$ induced by $P_1$ (it is the {\em induced}
branched cover since it branches over $\zeta_1(x_1)$ rather than over
$\zeta_1(0) = \omega_1$; see Section \ref{S:branched}).  Next, the
lifted marking identifies $\omega_2$ with a point in the preimage of
$\zeta_2(0)$, and we push this along an appropriate lift
$\zeta_2^{x_1}$ of $\zeta_2$ to a point $\zeta_2^{x_1}(x_2)$ in the
preimage of $\zeta_2(x_2)$.  At the next level, there is an branched
cover $\Omega_3^{x_1,x_2} \to (\Omega_2^{x_1},\zeta_2^{x_1}(x_2))$
induced by $P_2$.  The lifted marking identifies $\omega_3$ with a
point in the preimage of $\zeta_3(0)$ in the composition of branched
covers $\Omega_3^{x_1,x_2} \to \Omega_2^{x_1} \to \Omega_1$ and we
push this along an appropriate lift $\zeta_3^{x_1,x_2}$ of $\zeta_3$
to a point $\zeta_3^{x_1,x_2}(x_3)$ in the preimage of $\zeta_3(x_3)$.
We continue in this way to produce a tower of branched covers induced
by $P_1, P_2, \ldots, P_{n-3}, P_{n-2}$:
\[ 
\xymatrix{ \Omega_{n-1}^{x_1,\ldots,x_{n-2}} \ar[r] &
  \Omega_{n-2}^{x_1,\ldots,x_{n-3}} \ar[r] & \cdots \ar[r] &
  \Omega_3^{x_1,x_2} \ar[r] & \Omega_2^{x_1} \ar[r] & \Omega_1 .}
\]
The point $\omega_{n-1}$ is identified with a marked point in
$\Omega_{n-1}^{x_1,\ldots,x_{n-2}}$ in the preimage of
$\zeta_{n-1}(0)$, and then we push this point along an appropriate
lift $\zeta_{n-1}^{x_1,\ldots,x_{n-2}}$ of $\zeta_{n-1}$ to the point
$\zeta_{n-1}^{x_1,\ldots,x_{n-2}}(y) =
\zeta_{n-1}^{x_1,\ldots,x_{n-2}}(x_{n-1})$.  With this notation
\[ 
{\bf \Omega}(x,y,0) =
{\bf \Omega}(x_1,\ldots,x_{n-2},x_{n-1},0) =
(\Omega_{n-1}^{x_1,\ldots,x_{n-2}},\zeta_{n-1}^{x_1,\ldots,x_{n-2}}(x_{n-1})).
\]
To find ${\bf \Omega}(x,y,t)$ for any $t$, we apply
the appropriate Teichm\"uller deformation to ${\bf
\Omega}(x,y,0)$.  This is the Teichm\"uller deformation
determined by $t$ and the pull back of $\nu_1$ (via the composition of
branched covers).  We can pull back $\nu_1$ by any of the branched
covers, and since the resulting quadratic differential depends only on
the surface in this construction, we will simply write $\Phi_t$ for
the associated Teichm\"uller deformation on any of the surfaces
$\Omega_j^{x_1,\ldots,x_{j-1}}$.  In particular, we have
\[
{\bf \Omega}(x,y,t) = \Phi_t({\bf
  \Omega}(x,y,0)).
\]

Set $x' = (x_1', x_2', \ldots, x_{n-2}')$.  We now must find an
$\epsilon$ and $R$ so that if
\[ 
e^{-t}|(x,y)  - (x',y')| \leq \epsilon 
\]
then
\[ 
d_\T\left({\bf \Omega}(x,y,t),{\bf
  \Omega}(x',y',t)\right) \leq R. 
\]
As in the case of $\HH^3$, appealing to the triangle inequality it
suffices to find an $\epsilon$ and $R'$ so that if
$(x_1,\ldots,x_{n-2},y)$ and $(x_1',\ldots,x_{n-2}',y')$ agree in all
but one coordinate, and in that coordinate differ by at most
$\epsilon$, then
\[ 
d_\T\left({\bf \Omega}(x,y,t), {\bf \Omega}(x',y',t)\right) \leq R'. 
\]

If $(x,y)$ and $(x',y')$ differ only in the last coordinate, then we
can apply Theorem \ref{T:kra} just as before to produce $\epsilon = 1$
and $R' = A+B$.  Suppose instead that $y = y'$ and $x$ differs from
$x'$ in the $n-2$--coordinate only.  We start at the highest
coordinate, $y = x_{n-1}$ and work two steps down to $x_{n-2}$.  
The idea is similar to what was done in varying $x$ in $(x,y,t) \in
\HH^3$.  We look on $\Phi_t(\Omega_{n-2}^{x_1,\ldots,x_{n-3}})$ as an
``auxiliary'' surface when it is equipped with the {\em two} marked
points $\Phi_t(\zeta_{n-2}^{x_1,\ldots,x_{n-3}}(x_{n-2}))$ and the
image of $\Phi_t(\zeta_{n-1}^{x_1,\ldots,x_{n-2}}(x_{n-1}))$ via the
branched covering
\[ 
\Phi_t(\Omega_{n-1}^{x_1,\ldots,x_{n-2}}) \to
\Phi_t(\Omega_{n-2}^{x_1,\ldots,x_{n-3}}).
\]
If these two points are not too close, then we can move from
$\Phi_t(\zeta_{n-2}^{x_1,\ldots,x_{n-3}}(x_{n-2}))$ to
$\Phi_t(\zeta_{n-2}^{x_1,\ldots,x_{n-3}}(x_{n-2}'))$ keeping the other
marked point fixed, and the distance between these two points in the
Teichm\"uller space of the auxiliary surface with two marked points is
uniformly bounded.  Since the branched cover induces a $1$--Lipschitz
map (compare \eqref{E:branchedforgetbounds}), this means that
\[ 
d_\T\left({\bf \Omega}(x_1,\ldots,x_{n-3},x_{n-2},x_{n-1},t),{\bf
  \Omega}(x_1,\ldots,x_{n-3},x_{n-2}',x_{n-1},t)\right)
\]
is uniformly bounded.

On the other hand, if the two marked points in
$\Phi_t(\Omega_{n-2}^{x_1,\ldots,x_{n-3}})$ are close, we 
\[
\mbox{move} \quad \Phi_t(\zeta_{n-1}^{x_1,\ldots,x_{n-2}}(x_{n-1})) \quad
\mbox{to}   \quad \Phi_t(\zeta_{n-1}^{x_1,\ldots,x_{n-2}}(x_{n-1}')),
\]
\[
\mbox{move} \quad \Phi_t(\zeta_{n-2}^{x_1,\ldots,x_{n-3}}(x_{n-2})) \quad
\mbox{to}   \quad \Phi_t(\zeta_{n-2}^{x_1,\ldots,x_{n-3}}(x_{n-2}')),
\]
and then 
\[
\mbox{move}    \quad \Phi_t(\zeta_{n-1}^{x_1,\ldots,x_{n-2}'}(x_{n-1}')) \quad 
\mbox{back to} \quad \Phi_t(\zeta_{n-1}^{x_1,\ldots,x_{n-2}'}(x_{n-1})).
\]
By the triangle inequality, we obtain the desired uniform bound on
\[ 
d_\T\left({\bf \Omega}(x_1,\ldots,x_{n-3},x_{n-2},x_{n-1},t),{\bf
  \Omega}(x_1,\ldots,x_{n-3},x_{n-2}',x_{n-1},t)\right).
\]
Note that this required three point pushes in two different auxiliary
surfaces.  We varied the $(n-1)^{\rm st}$ coordinate twice, in
the highest surface, and varied the $(n-2)^{\rm nd}$ coordinate once.

Now suppose that $x$ differs from $x'$ in the $(n-3)^{\rm rd}$
coordinate only.  We view $\Phi_t(\Omega_{n-3}^{x_1,\ldots,x_{n-4}})$
as an auxiliary surface with {\em three marked points}: the images of
the points $\Phi_t(\zeta_{n-1}^{x_1,\ldots,x_{n-2}}(x_{n-1}))$ and
$\Phi_t(\zeta_{n-2}^{x_1,\ldots,x_{n-3}}(x_{n-2}))$ under the
respective branched covers and the point
$\Phi_t(\zeta_{n-3}^{x_1,\ldots,x_{n-4}}(x_{n-3}))$.  We can move this
last point a small amount, changing the Teichm\"uller distance a
bounded amount, provided the other two points, higher in the tower,
are not too close to it.  If they are too close, we first move them
out of the way (as in the first two pushes above), move the third
point, then move the two higher points back.  The triangle inequality
together with the $1$--Lipschitz property of the branched cover map
applied as before, implies a uniform bound on the change in
Teichm\"uller distance
\[ 
d_\T\left({\bf \Omega}(x_1,\ldots,x_{n-3},x_{n-2},x_{n-1},t),{\bf
  \Omega}(x_1,\ldots,x_{n-3}',x_{n-2},x_{n-1},t)\right).
\]
It follows that varying $x_{n-3}$ requires at most five point pushes
in the three highest auxiliary surfaces.

In general, varying $x_{n-k}$ in this way requires $2k - 1$ point
pushes in the $k$ highest auxiliary surfaces.  Thus we can change any
coordinate by a small amount $\epsilon$ and change the Teichm\"uller
distance by a bounded amount $R'$, as required. This completes the
sketch of the proof of \refthm{Hn_isometric}.
\end{proof}

\bibliography{hyp_teich}{}
\bibliographystyle{plain}

\end{document}